\documentclass[a4paper,10pt]{amsart}
\usepackage{amsmath,amsthm,amssymb,latexsym,enumerate,color,hyperref}
\usepackage{graphicx}

\numberwithin{equation}{section}

\begin{document}

\newtheorem{thm}{Theorem}[section]
\newtheorem{prop}[thm]{Proposition}
\newtheorem{lem}[thm]{Lemma}
\newtheorem{cor}[thm]{Corollary}
\newtheorem{rem}[thm]{Remark}
\newtheorem*{defn}{Definition}

\newtheorem{SBT}{Theorem}
\renewcommand*{\theSBT}{\Alph{SBT}}

\newcommand{\DD}{\mathbb{D}}
\newcommand{\NN}{\mathbb{N}}
\newcommand{\ZZ}{\mathbb{Z}}
\newcommand{\QQ}{\mathbb{Q}}
\newcommand{\RR}{\mathbb{R}}
\newcommand{\CC}{\mathbb{C}}
\renewcommand{\SS}{\mathbb{S}}

\renewcommand{\theequation}{\arabic{section}.\arabic{equation}}

\newcommand{\supp}{\mathop{\mathrm{supp}}}    

\newcommand{\re}{\mathop{\mathrm{Re}}}   
\newcommand{\im}{\mathop{\mathrm{Im}}}   
\newcommand{\dist}{\mathop{\mathrm{dist}}}  
\newcommand{\link}{\mathop{\circ\kern-.35em -}}
\newcommand{\spn}{\mathop{\mathrm{span}}}   
\newcommand{\ind}{\mathop{\mathrm{ind}}}   
\newcommand{\rank}{\mathop{\mathrm{rank}}}   
\newcommand{\Fix}{\mathop{\mathrm{Fix}}}   
\newcommand{\codim}{\mathop{\mathrm{codim}}}   
\newcommand{\conv}{\mathop{\mathrm{conv}}}   
\newcommand{\epsi}{\mbox{$\varepsilon$}}
\newcommand{\eps}{\mathchoice{\epsi}{\epsi}
{\mbox{\scriptsize\epsi}}{\mbox{\tiny\epsi}}}
\newcommand{\cl}{\overline}
\newcommand{\pa}{\partial}
\newcommand{\ve}{\varepsilon}
\newcommand{\zi}{\zeta}
\newcommand{\Si}{\Sigma}
\newcommand{\cA}{{\mathcal A}}
\newcommand{\cG}{{\mathcal G}}
\newcommand{\cH}{{\mathcal H}}
\newcommand{\cI}{{\mathcal I}}
\newcommand{\cJ}{{\mathcal J}}
\newcommand{\cK}{{\mathcal K}}
\newcommand{\cL}{{\mathcal L}}
\newcommand{\cN}{{\mathcal N}}
\newcommand{\cR}{{\mathcal R}}
\newcommand{\cS}{{\mathcal S}}
\newcommand{\cT}{{\mathcal T}}
\newcommand{\cU}{{\mathcal U}}
\newcommand{\OM}{\Omega}
\newcommand{\B}{\bullet}
\newcommand{\ol}{\overline}
\newcommand{\ul}{\underline}
\newcommand{\vp}{\varphi}
\newcommand{\AC}{\mathop{\mathrm{AC}}}   
\newcommand{\Lip}{\mathop{\mathrm{Lip}}}   
\newcommand{\es}{\mathop{\mathrm{esssup}}}   
\newcommand{\les}{\mathop{\mathrm{les}}}   
\newcommand{\nid}{\noindent}
\newcommand{\pzr}{\phi^0_R}
\newcommand{\pir}{\phi^\infty_R}
\newcommand{\psr}{\phi^*_R}
\newcommand{\pow}{\frac{N}{N-1}}
\newcommand{\ncl}{\mathop{\mathrm{nc-lim}}}   
\newcommand{\nvl}{\mathop{\mathrm{nv-lim}}}  
\newcommand{\la}{\lambda}
\newcommand{\La}{\Lambda}    
\newcommand{\de}{\delta}    
\newcommand{\fhi}{\varphi} 
\newcommand{\ga}{\gamma}    
\newcommand{\ka}{\kappa}   

\newcommand{\core}{\heartsuit}
\newcommand{\diam}{\mathrm{diam}}

\newcommand{\lan}{\langle}
\newcommand{\ran}{\rangle}
\newcommand{\tr}{\mathop{\mathrm{tr}}}
\newcommand{\diag}{\mathop{\mathrm{diag}}}
\newcommand{\dv}{\mathop{\mathrm{div}}}

\newcommand{\al}{\alpha}
\newcommand{\be}{\beta}
\newcommand{\Om}{\Omega}
\newcommand{\na}{\nabla}

\newcommand{\cC}{\mathcal{C}}
\newcommand{\cM}{\mathcal{M}}
\newcommand{\nr}{\Vert}
\newcommand{\De}{\Delta}
\newcommand{\cX}{\mathcal{X}}
\newcommand{\cP}{\mathcal{P}}
\newcommand{\om}{\omega}
\newcommand{\si}{\sigma}
\newcommand{\te}{\theta}
\newcommand{\Ga}{\Gamma}

\newcommand{\pCap}{\operatorname{Cap}}

\title[Radial symmetry for $p$-harmonic functions]{Radial symmetry for $p$-harmonic functions in exterior and punctured domains}

\author{Giorgio Poggesi}
\address{Dipartimento di Matematica ed Informatica ``U.~Dini'',
Universit\` a di Firenze, viale Morgagni 67/A, 50134 Firenze, Italy.}
       \email{giorgio.poggesi@unifi.it}





\begin{abstract}
We prove symmetry for the $p$-capacitary potential satisfying
\begin{equation*}
\De_p u = 0 \, \mbox{ in } \RR^N \setminus \ol{\Om} , \; u=1 \, \mbox{ on } \Ga, \; \lim_{|x|\rightarrow \infty} u(x)=0 , \; \; \; \; \; \; \; \; 1<p<N,
\end{equation*}
under Serrin's overdetermined condition
$$
| \na u| = c \mbox{ on } \Ga.
$$
Here $\Om$ is any bounded domain on which no a priori assumption is made, and $\Ga$ denotes its boundary. Our result improves on the work \cite{GS}, where the same conclusion was obtained when $\Om$ is star-shaped.
Our proof uses the maximum principle for an appropriate $P$-function, some integral identities, the isoperimetric inequality, and a Soap Bubble-type Theorem.

We then 
treat the case $1<p=N$, improving previous results present in the literature.

Finally, with analogous tools we give a new proof of symmetry for the interior overdetermined problem
\begin{equation*}
- \De_p u = K \, \de_0 \, \mbox{ in } \Om , \, u=c \, \mbox{ on } \Ga, \; \; \; \; \; \; \; \; 1<p<N,
\end{equation*} 
\begin{equation*}
| \na u| = 1  \mbox{ on } \Ga ,
\end{equation*}
in a bounded star-shaped domain $\Om$.

\end{abstract}

\maketitle

\raggedbottom

\section{Introduction}

The \textit{electrostatic $p$-capacity} of a bounded open set $\Om \subset \RR^N$, $1<p<N$, is defined by
\begin{equation}\label{def:Capacity}
\pCap_p(\Om)= \inf \left\lbrace \int_{\RR^N} | \na w|^p : \; w \in C^{\infty}_{0} (\RR^N), \; w \ge 1 \; \text{in} \; \Om \right\rbrace
\end{equation}

Under appropriate sufficient conditions, there exists a unique minimizing function $u$ of \eqref{def:Capacity}; such function is called the {\it $p$-capacitary potential} of $\Om$, and satisfies
\begin{equation}
\label{Problemcapacity}
\De_p u = 0 \, \mbox{ in } \RR^N \setminus \ol{\Om} , \; u=1 \, \mbox{ on } \Ga, \; \lim_{|x|\rightarrow \infty} u(x)=0 ,
\end{equation}
where $\Ga$ is the boundary of $\Om$ and $\De_p$ denotes the $p$-Laplace operator defined by
\begin{equation*}
\De_p u = \dv \left( |\na u|^{p-2} \na u \right).
\end{equation*}


It is well known that the $p$-capacity could be equivalently defined by means of the $p$-capacitary potential $u$ as
\begin{equation}\label{eq:caponboundary}
\pCap_p(\Om) = \int_{\RR^N \setminus \ol{\Om} } |\na u|^p \, dx = \int_\Ga | \na u|^{p-1} \, dS_x,
\end{equation}
where the second equality follows by integration by parts (see the proof of Lemma~\ref{lem:asymptoticcap}).

In Section \ref{sec:exterior} we consider Problem \eqref{Problemcapacity} under Serrin's overdetermined condition given by
\begin{equation}\label{eq:overdetermination}
| \na u| = c  \mbox{ on } \Ga,
\end{equation}
and we prove the following result.

\begin{thm}\label{thm:Serrinesterno}
Let $\Om \subset \RR^N$ be a bounded domain (i.e., a connected bounded open set).
For $1<p<N$, the problem \eqref{Problemcapacity}, \eqref{eq:overdetermination} admits a weak solution if and only if $\Om$ is a ball.
\end{thm}


By a weak solution in the statement of Theorem \ref{thm:Serrinesterno} we mean a function $u \in W_{loc}^{1,p} \left( \RR^N \setminus \ol{\Om} \right)$ such that
$$
\int_{\RR^N \setminus \ol{\Om}} |\na u|^{p-2} \, \na u \cdot \na \phi \, dx=0
$$
for every $\phi \in C_0^\infty \left( \RR^N \setminus \ol{\Om} \right)$
and satisfying the boundary conditions in the weak sense, i.e.: for all $\ve >0$ there exists a neighborhood $U_\ve \supset \Ga$ such that $|u(x)- 1| < \ve$ and $\left| |\na u| - c \right| < \ve$ for a.e. $x \in U_\ve \cap \left( \RR^N \setminus \ol{\Om} \right)$.

Notice the complete absence in Theorem \ref{thm:Serrinesterno} of any smoothness assumption as well as of any other assumption on the domain $\Om$.

Under the additional assumption that $\Om$ is star-shaped, Theorem \ref{thm:Serrinesterno} has been proved by Garofalo and Sartori in \cite{GS}, by extending to the case $1<p<N$ the tools developed in the case $p=2$ by Payne and Philippin (\cite{PP}, \cite{Ph}).
The proof in \cite{GS}, which combines integral identities and a maximum principle for an appropriate $P$-function, bears a resemblance to Weinberger's proof (\cite{We}) of symmetry for the archetype {\it torsion problem}
\begin{equation}\label{eq:torsionproblem}
\De \tau =N \text{ in } \Om, \quad \tau =0 \text{ on } \Ga ,
\end{equation}
under Serrin's overdetermined condition
\begin{equation}\label{eq:torsionoverdet}
| \na \tau|= const. \text{ on }  \Ga.
\end{equation}

%
%
%
%

To prove Theorem \ref{thm:Serrinesterno}, we improve on the arguments used in \cite{GS} and we exploit a new crucial ingredient, that is the following Soap Bubble-type Theorem proved via integral identities by Magnanini and the author in \cite[Theorem 2.2]{MP} (see also \cite[Theorem 2.4]{MP2}).
By $| \Om|$ and $| \Ga|$ we will denote the $N$-dimensional Lebesgue measure of $\Om$ and the surface measure of $\Ga$.
\begin{SBT}[\cite{MP}, \cite{MP2}]\label{provaSBT}
Let $\Ga$ be a $C^2$-surface, which is the boundary of a bounded domain $\Om \subset \RR^N$ and let $H_0$ be the constant defined by
\begin{equation}\label{def:H0}
H_0 = \frac{| \Ga|}{N | \Om|}.
\end{equation}
Let $\tau$ be the solution of the torsion problem \eqref{eq:torsionproblem}.
If the mean curvature $H$ of $\Ga$ satisfies the inequality
$$\int_\Ga (H_0 - H) \,  |\na \tau |^2  \, dS_x \le 0 ,$$
then $\Ga$ must be a sphere (and hence $\Om$ is a ball) of radius $\frac{1}{H_0}$.
In particular, the same conclusion holds if
\begin{equation*}
H \ge H_0 \text{ on } \Ga .
\end{equation*}
\end{SBT}

%


Symmetry for Problem \eqref{Problemcapacity}, \eqref{eq:overdetermination} was first obtained by Reichel (\cite{Re1}, \cite{Re2}) by adapting the {\it method of moving planes} introduced by Serrin (\cite{Se}) to prove symmetry for the overdetermined torsion problem \eqref{eq:torsionproblem}, \eqref{eq:torsionoverdet}.
In Reichel's works (\cite{Re1}, \cite{Re2}) the star-shapedness assumption is not requested, but the domain $\Om$ is a priori assumed to be $C^{2,\al}$ and the solution $u$ is assumed to be of class $C^{1, \al}(\RR^N \setminus \Om)$.

For completeness let us mention that many alternative proofs and improvements in various directions of symmetry results for Serrin's problems relative to the equations \eqref{eq:torsionproblem} and \eqref{Problemcapacity} have been obtained in the years and can be found in the literature: for the overdetermined torsion problem \eqref{eq:torsionproblem}, \eqref{eq:torsionoverdet} see for example \cite{PS}, \cite{GL}, \cite{DP}, \cite{BH}, \cite{BNST}, \cite{FK}, \cite{CiS}, \cite{WX}, \cite{MP2}, \cite{BC}, and the surveys \cite{Mag}, \cite{NT}, \cite{Ka}; for the exterior overdetermined problem \eqref{Problemcapacity}, \eqref{eq:overdetermination} where the domain $\Om$ is assumed to be convex see \cite{MR}, \cite{BCS} and \cite{BC}.


In Section \ref{sec:casoconforme}, we establish the result corresponding to Theorem \ref{thm:Serrinesterno} in the special case $1<p=N$.
In this case, the problem corresponding to \eqref{Problemcapacity} is (see e.g. \cite{CC}):
\begin{equation}
\label{ConformalProblemcapacity}
\De_N u = 0 \, \mbox{ in } \RR^N \setminus \ol{\Om} , \; u=1 \, \mbox{ on } \Ga, \; u(x) \sim - \ln{|x|}  \text{ as } |x|\to \infty ,
\end{equation}
where $\sim$ means that
\begin{equation}\label{eq:confasymppre}
c_1 \le \frac{u(x)}{ ( - \ln{|x|} ) } \le c_2, \text{ as } |x| \to \infty
\end{equation}
for some positive constants $c_1$,$c_2$.
What we prove is the following.
\begin{thm}\label{thm:conforme}
Let $\Om \subset \RR^N$, $1<N$, be a bounded domain.
The problem \eqref{ConformalProblemcapacity} with the overdetermined condition \eqref{eq:overdetermination}
admits a weak solution if and only if $\Om$ is a ball.
\end{thm}

A proof of Theorem \ref{thm:conforme} that uses the method of moving planes is contained in \cite{Re2}, under the additional a priori smoothness assumptions $\Ga \in C^{2, \al}$, $u \in C^{1,\al}(\RR^N \setminus \Om)$.
Our proof via integral identities seems to be new and cannot be found in the literature unless for the classical case $p=N=2$, which has been treated with similar arguments in \cite{Mar} (for piecewise smooth domains) and \cite{MR} (for Lipschitz domains). Moreover, in Theorem \ref{thm:conforme} no assumptions on the domain $\Om$ are made.

We mention that a related symmetry result for the $N$-capacitary potential in a bounded (smooth) star-shaped ring domain has been established in \cite{PP2}.


In Section \ref{sec:Interior}, we show how the same ideas used in our proof of Theorem \ref{thm:Serrinesterno} can be adapted to give a symmetry result for a similar problem in a bounded punctured domain. More precisely, we prove the following theorem concerning the problem
\begin{equation}\label{Pdelta}
- \De_p u = K \, \de_0 \, \mbox{ in } \Om , \, u=c \, \mbox{ on } \Ga,
\end{equation} 
under Serrin's overdetermined condition 
\begin{equation}\label{eq:interno-overdetermination}
| \na u| = 1  \mbox{ on } \Ga ,
\end{equation}
where with $\de_0$ we denote the Dirac delta centered at the origin $0 \in \Om$ and $K$ is some positive normalization constant.


\begin{thm}
\label{thm:Serrininterno}
Let $\Om \subset \RR^N$ be a bounded star-shaped domain.
For $1<p<N$, the problem \eqref{Pdelta}, \eqref{eq:interno-overdetermination}
admits a weak solution if and only if $\Om$ is a ball centered at the origin.
\end{thm}

It should be noticed, however, that in Theorem \ref{thm:Serrininterno} we need to assume $\Om$ to be star-shaped, restriction that is not present in the proofs of Payne and Schaefer (\cite{PS})(for the case $p=2$), Alessandrini and Rosset (\cite{AR}), and Enciso and Peralta-Salas (\cite{EP}).
We mention that the proof in \cite{AR} uses an adaptation of the method of moving planes, the proof in \cite{EP} is in the wake of Weinberger, and both of them also cover the special case $p=N$.



For $p=2$, Problem \eqref{Problemcapacity} arises naturally in electrostatics. In this context $u$ is the (normalized) potential of the electric field $\na u$ generated by a conductor $\Om$.
We recall that when $\Om$ is in the electric equilibrium, the electric field in the interior of $\Om$ is null, and hence the electric potential $u$ is constant in $\Om$ (i.e. $u \equiv 1$ in $\Om$); moreover, the electric charges present in the conductor are distributed on the boundary $\Ga$ of $\Om$.
It is also known that the electric field $\na u$ on $\Ga$ is orthogonal to $\Ga$ -- i.e. $\na u = u_\nu \, \nu$, where here as in the rest of the paper $\nu$ denotes the outer unit normal with respect to $\Om$ and $u_\nu$ denotes the derivative of $u$ in the direction $\nu$
-- and its intensity is given\footnote{More precisely, in free space the intensity of the electric field on $\Ga$ is given by $|\na u|_{| \Ga} = - u_\nu = \frac{\rho}{\eps_0}$, where $\rho$ is the surface charge density over $\Ga$ and $\eps_0$ is the vacuum permittivity. We ignored the constant $\eps_0$ to be coherent with the mathematical definition of capacity given in \eqref{def:Capacity}.} by the surface charge density over $\Ga$. In this context the capacity is defined as the total electric charge needed to induce the potential $u$, that is 
$$\pCap(\Om)= \int_\Ga | \na u | \, dS_x, $$
in accordance with \eqref{eq:caponboundary} for $p=2$.
Following this physical interpretation Theorem \ref{thm:Serrinesterno} simply states that the electric field on the boundary $\Ga$ of the conductor $\Om$ is constant - or equivalently that the charges present in the conductor are uniformly distributed on $\Ga$ (i.e. the surface charge density is constant over $\Ga$)- if and only if $\Om$ is a round ball.


Another result of interest in the same context that is related to Problem \eqref{Problemcapacity}, \eqref{eq:overdetermination} is a Poincar\'e's theorem known as the isoperimetric inequality for the capacity, stating that, among sets having given volume, the ball minimizes $\pCap(\Om)$.
We mention that a proof of this inequality that hinges on rearrangement techniques can be found in \cite[Section 1.12]{PZ} (see also \cite{Ja} for a useful review of that proof).
Here, we just want to underline the strong relation present between this result and Problem \eqref{Problemcapacity}, \eqref{eq:overdetermination} (for $p=2$).
In fact, once that the existence of a minimizing set $\Om_0$ is established, we can show through the technique of {\it shape derivative} that the solution of \eqref{Problemcapacity} in $\Om_0$ also satisfies the overdetermined condition \eqref{eq:overdetermination} on $\Ga_0=\pa \Om_0$; the reasoning is the following.


We consider the evolution of the domains $\Om_t$ given by
\begin{equation*}
\Om_t = \cM_t (\Om), 
\end{equation*}
where $\Om= \Om_0$ is fixed and $\cM_t: \RR^N \to \RR^N$ is a mapping such that
$$
\cM_0 (x) =x, \; \cM'_0 (x) = \phi(x) \nu(x),
$$
where the symbol $'$ means differentiation with respect to $t$, $\phi$ is any compactly supported continuous function, and $\nu$ is a proper extension of the unit normal vector field to a tubular neighborhood of $\Ga_0$ (for instance $\nu(x)= \na \de_{\Ga_0} (x)$, where $\de_{\Ga_0} (x)$ is the distance of $x$ from $\Ga_0$).
Thus, we consider $u(t,x)$, solution of Problem \eqref{Problemcapacity} in $\Om= \Om_t$, and the two functions (in the variable $t$) $\pCap (\Om_t)$ and $| \Om_t|$.
%
%
%
Since $\Om_0$ is the domain that minimizes $\pCap(\Om_t)$ among all the domains in the one-parameter family $\left\lbrace \Om_t \right\rbrace_{t \in \RR}$ that have prescribed volume $|\Om_t|=V$, by using the method of {\it Lagrange multipliers} and {\it Hadamard's variational formula} (see \cite[Chapter 5]{HP}), standard computations lead to prove that there exists a number $\la$ such that
\begin{equation*}
\int_{\Ga_0} \phi(x) \left[ u_\nu^2 (x) - \la \right] \, dS_x = 0,
\end{equation*}
where we have set $u(x)=u(0,x)$.
Since $\phi$ is arbitrary, we deduce that $u_\nu^2 \equiv \la$ on $\Ga_0$, that is, $u$ satisfies the overdetermined condition \eqref{eq:overdetermination} on $\Ga_0$.



Other applications of Problem \eqref{Problemcapacity} are related to quantum theory, acoustic, theory of musical instruments, and the study of heat, electrical and fluid flow (see for example \cite{CFG}, \cite{DZH}, \cite{BC} and references therein).

Finally, we just mention that also Problem \eqref{Pdelta} arises in electrostatics for $p=2$: in this case, Theorem \ref{thm:Serrininterno} states that the electric field on a conducting hypersurface enclosing a charge is constant if and only if the conductor is a sphere centered at the charge.

\section{The exterior problem: proof of Theorem \ref{thm:Serrinesterno}}\label{sec:exterior}
In order to prove Theorem \ref{thm:Serrinesterno}, we start by collecting all the necessary ingredients.
In this section $u$ denotes a weak solution to \eqref{Problemcapacity},\eqref{eq:overdetermination}, in the sense explained in the Introduction, and it holds that $1<p<N$.



\begin{rem}[On the regularity]\label{rem:regularity}
{\rm
Due to the degeneracy or singularity of the $p$-Laplacian (when $p \neq 2$) at the critical points of $u$, $u$ is in general only $C^{1, \al}_{loc}(\RR^N \setminus \ol{\Om})$ (see \cite{Di}, \cite{Le}, \cite{To}), whereas it is $C^{\infty}$ in a neighborhood of any of its regular points thanks to standard elliptic regularity theory (see \cite{GT}). However, as already noticed in \cite{GS}, the additional assumption given by the weak boundary condition \eqref{eq:overdetermination} ensures that $u$ can be extended to a $C^2$-function in a neighborhood of $\Ga$, so that by using the work of Vogel \cite{Vo} we get that $\Ga$ is of class $C^2$. Thus, by \cite[Theorem 1]{Li} it turns out that $u$ is $C^{1, \al}_{loc} (\RR^N \setminus \Om)$.
As a consequence we can now interpret the boundary condition in \eqref{Problemcapacity} and the one in \eqref{eq:overdetermination} in the classical strong sense.
}
\end{rem}

\begin{rem}
{\rm
More precisely, by using Vogel's work \cite{Vo}, which is based on the deep results on free boundaries contained in \cite{AC} and \cite{ACF}, one can prove that $\Ga$ is of class $C^{2,\al}$ from each side. Even if in the present paper we do not need this refinement, it should be noticed that in light of this remark the arguments contained in \cite{Re2} give an alternative and complete proof of Theorem \ref{thm:Serrinesterno} (and also of Theorem \ref{thm:conforme}). In fact, the smoothness assumptions of \cite[Theorem 1]{Re2} are satisfied.
}
\end{rem}


By using the ideas contained in \cite{GS} together with a result of Kichenassamy and V\'eron (\cite{KV}), we can recover the following useful asymptotic expansion for $u$ as $|x|$ tends to infinity.
\begin{lem}[Asymptotic expansion]\label{lem:asymptoticcap}
As $|x|$ tends to infinity it holds that
\begin{equation}\label{eq:cap-asymptotic u}
u(x) = \frac{p-1}{N-p} \left( \frac{\pCap_p (\Om)}{\om_N} \right)^{\frac{1}{p-1}} |x|^{- \frac{N-p}{p-1}} + o (|x|^{- \frac{N-p}{p-1}}).
\end{equation}
\end{lem}


The computations leading to determine the constant of proportionality in \eqref{eq:cap-asymptotic u} originate from \cite{GS}, where they have been used to give a complete proof -- which works without invoking \cite{KV} -- of a different but related result
(\cite[Theorem~3.1]{GS}), which is described in more details in Remark \ref{rem:Garofaoloasymptotic} below.
We mention that, later, in \cite{CoS} the same computations have been treated with tools of convex analysis and used together with the result contained in \cite[Remark 1.5]{KV} to prove Lemma \ref{lem:asymptoticcap} for convex sets.

\begin{proof}[Proof of Lemma \ref{lem:asymptoticcap}]
As noticed in \cite{GS}, if $u$ is a solution of \eqref{Problemcapacity}, then the weak comparison principle for the $p$-Laplacian (see \cite{HKM}) implies the existence of positive constants $c_1$,$c_2$, $R_0$ such that
$$
c_1 \mu(x) \le u(x) \le c_2 \mu(x), \quad \text{ if } |x| \ge R_0,
$$
where $\mu(x)$ denotes the radial fundamental solution of the $p$-Laplace operator given by
\begin{equation*}
\mu(x) = \frac{(p-1)}{(N-p)} \frac{1}{\om_N^{\frac{1}{p-1}}} |x|^{\frac{p-N}{p-1}}.
\end{equation*}
Thus we can apply the result of Kichenassamy and V\'eron (\cite[Remark 1.5]{KV}) and state that there exists a constant $\ga$ such that
\begin{equation}\label{eq:passKV}
\lim_{|x| \to \infty} \frac{u(x)}{\mu(x)} = \ga, \quad \lim_{|x| \to \infty} |x|^{\frac{N-p}{p-1} + | \al|}  D^\al \left( u - \ga \mu \right)=0,
\end{equation}
for all multi-indices $\al=(\al_1, \dots, \al_N)$ with $|\al|= \al_1+ \dots + \al_N \ge 1$.
To establish \eqref{eq:cap-asymptotic u} now it is enough to prove that
\begin{equation}\label{eq:0gamma}
\ga = \pCap_p (\Om)^{\frac{1}{p-1}};
\end{equation}
this can be easily done, as already noticed in \cite{GS}, through the following integration by parts that holds true by the $p$-harmonicity of $u$:
\begin{equation}\label{eq:capintegrperasymptotic}
- \int_{\Ga} |\na u|^{p-2} u_{\nu} \, dS_x = - \lim_{R \to \infty} \int_{\pa B_R} |\na u|^{p-2} u_{\nu_{B_R}} \, dS_x .
\end{equation}
Here as in the rest of the paper $\nu_{B_R}$ denotes the outer unit normal with respect to the ball $B_R$ of radius $R$, $\nu$ is the outer unit normal with respect to $\Om$, and $u_\nu$ (resp.
$u_{\nu_{B_R}}$) is the derivative of $u$ in the direction $\nu$ (resp. $\nu_{B_R}$). 
The left-hand side of \eqref{eq:capintegrperasymptotic} is exactly $\pCap_p (\Om)$ as it is clear by \eqref{eq:caponboundary} and the fact that 
\begin{equation}\label{eq:capnormagradmenodernormale}
| \na u|= - u_\nu \text{ on } \Ga ;
\end{equation}
moreover, the limit in the right-hand side of \eqref{eq:capintegrperasymptotic} can be explicitly computed by using the second equation in \eqref{eq:passKV} (with $|\al|=1$) and it turns out to be $\ga^{p-1}$. Thus, \eqref{eq:0gamma} is proved and \eqref{eq:cap-asymptotic u} follows.

For completeness, we explain here how to prove the second identity in \eqref{eq:caponboundary}: we take the limit for $R \to \infty$ of the following integration by parts made on $B_R \setminus \ol{\Om} $ and we note that the integral on $\pa B_R$ converge to zero due to
\eqref{eq:passKV}:
\begin{equation*}
\int_{B_R \setminus \ol{\Om} } |\na u|^p \, dx = \int_{\pa B_R} u | \na u|^{p-2} u_{\nu_{B_R}} \, dS_x - \int_{\Ga} |\na u|^{p-2} u_\nu \, dS_x .
\end{equation*}
The desired identity is thus proved just by recalling \eqref{eq:capnormagradmenodernormale}.
\end{proof}


It is well known that the value $c$ of $| \na u|$ on $ \Ga$ appearing in the overdetermined condition \eqref{eq:overdetermination} can be explicitly computed.

\begin{lem}[Explicit value of $c$ in \eqref{eq:overdetermination}]
The constant $c$ appearing in \eqref{eq:overdetermination} equals
\begin{equation}\label{eq:valoreunu serrinesterno}
c = \frac{N-p}{p-1} \, H_0,
\end{equation}
where $H_0$ is the constant defined in \eqref{def:H0}.
Moreover, the following explicit expression of the $p$-capacity of $\Om$ holds:
\begin{equation}\label{eq:cap in terms isop}
\pCap_p (\Om) = \left( \frac{N-p}{p-1} \right)^{p-1} \frac{| \Ga|^p}{(N | \Om|)^{p-1}}.
\end{equation}
\end{lem}
\proof
It is enough to use \eqref{eq:caponboundary} together with the following Rellich-Pohozaev-type identity:
\begin{equation}\label{eq:Poho}
(N-p)\int_{\RR^N \setminus \ol{\Om} } |\na u|^p \, dx = (p-1)\int_\Ga |\na u|^p <x, \nu> \, dS_x .
\end{equation}
In fact, by using \eqref{eq:overdetermination} in \eqref{eq:caponboundary} and \eqref{eq:Poho} we deduce respectively that
\begin{equation*}
\pCap_p (\Om) = c^{p-1} | \Ga| \quad \text{and} \quad \pCap_p (\Om)=\frac{p-1}{N-p} c^p N | \Om| ,
\end{equation*}
from which we get \eqref{eq:valoreunu serrinesterno} and \eqref{eq:cap in terms isop}.

Equation \eqref{eq:Poho} comes directly by taking the limit for $R \to \infty$ of the following integration by parts made on $B_R \setminus \ol{\Om} $ and noting that the integrals on $\pa B_R$ converge to zero due to the asymptotic going of $u$ at infinity given by \eqref{eq:cap-asymptotic u}:
\begin{multline}\label{eq:provaconforme}
(N-p)\int_{B_R \setminus \ol{\Om} } |\na u|^p \, dx = 
\\
p \int_\Ga |\na u|^{p-2} <x, \na u> u_\nu \, dS_x - \int_\Ga |\na u|^p <x, \nu> \, dS_x -
\\
 p \int_{\pa B_R } |\na u|^{p-2} <x, \na u> u_{\nu_{B_R}} \, dS_x + \int_{\pa B_R } |\na u|^p <x, \nu_{B_R}> \, dS_x .
\end{multline}
Equation \eqref{eq:Poho} is in fact proved just by recalling that 
\begin{equation}\label{eq:capgraddernormnormale}
\na u = u_\nu \, \nu \text{ on }\Ga.
\end{equation}
\endproof


\vspace{1pt}
\noindent
\textbf{The P-function.}
As last ingredient, we introduce the $P$-function
\begin{equation}
\label{def:P}
P= \frac{|\na u|^p}{u^{\frac{p(N-1)}{(N-p)}}}.
\end{equation}
Notice that in the radial case, i.e. if $\Om= B_R (x_0)$ is a ball of radius $R$ centered at the point $x_0$, we have that
\begin{equation}
\label{eq:esplicitaradiale1}
u(x) = \left( \frac{R}{|x  - x_0  |} \right)^{\frac{N-p}{p-1}}
\end{equation}
and thus
$$
P \equiv \left( \frac{N-p}{p-1} \right)^{p} R^{-p}.
$$

In \cite{GS} the authors have studied extensively the properties of the function $P$.
In particular, in \cite[Theorem 2.2]{GS} it is proved that $P$ satisfies the strong maximum principle, i.e. the function $P$ cannot attain a local maximum at an interior point of $\RR^N \setminus \ol{ \Om}$, unless $P$ is constant. We mention that this property for the case $p=2$ was first established in \cite{PP}.

\vspace{1pt}

Now that we collected all the ingredients, we are in position to give the proof of Theorem \ref{thm:Serrinesterno}.


\begin{proof}[Proof of Theorem \ref{thm:Serrinesterno}]

By \eqref{eq:cap-asymptotic u} it is easy to check that
\begin{equation*}
\lim_{|x| \rightarrow \infty} P(x)= \left( \frac{N-p}{p-1} \right)^{\frac{p(N-1)}{N-p}} \left( \frac{ \om_N }{ \pCap_p(\Om)} \right)^{\frac{p}{N-p}} ,
\end{equation*}
from which, by using \eqref{eq:cap in terms isop}, we get
\begin{equation}\label{eq:prova2}
\lim_{|x| \rightarrow \infty} P(x)= \left( \frac{N-p}{p-1} \right)^{p} \left( \frac{ \om_N \left( N | \Om| \right)^{p-1} }{ | \Ga|^p } \right)^{\frac{p}{N-p}}.
\end{equation}

Moreover, by recalling the boundary condition in 
\eqref{Problemcapacity}, \eqref{eq:overdetermination}, and \eqref{eq:valoreunu serrinesterno}
we can compute that
\begin{equation}\label{eq:prova}
P_{| \Ga} =  \left( \frac{N-p}{p-1} \right)^p  \left( \frac{| \Ga| }{ N | \Om| } \right)^p.
\end{equation}

By using the classical isoperimetric inequality (see, e.g., \cite{BZ})
\begin{equation}\label{eq:ISOPerimetrica}
\frac{| \Ga|^{\frac{N}{N-1}}}{N \om_N^{\frac{1}{N-1}}} \ge |\Om|,
\end{equation}
by \eqref{eq:prova2} and \eqref{eq:prova} it is easy to check that
$$
\lim_{|x| \rightarrow \infty} P(x) \le P_{| \Ga}.
$$
Hence, by the strong maximum principle proved in \cite[Theorem 2.2]{GS}, $P$ attains its maximum on $\Ga$ and thus we can affirm that
\begin{equation}\label{eq:00}
P_{\nu} \le 0 ,
\end{equation}
where, $\nu$ is still the outer unit normal with respect to $\Om$.
If we directly compute $P_{\nu}$ and we use \eqref{eq:capgraddernormnormale}, we find
\begin{equation}\label{eq:01}
P_\nu = p \, u^{- \frac{p(N-1)}{N-p}} |\na u|^{p-2}  \left\lbrace u_{\nu \nu} u_\nu - \frac{N-1}{N-p} \, |\na u|^2 \, \frac{u_\nu}{u} \right\rbrace.
\end{equation}
By the well known differential identity
\begin{equation*}
\De_p u = | \na u|^{p-2} \left\lbrace (p-1) u_{\nu \nu} + (N-1) H u_\nu \right\rbrace \quad \mbox{ on $\Ga$    }
\end{equation*}
and the $p$-harmonicity of $u$ we deduce that
\begin{equation}\label{eq:02}
u_{\nu \nu}= - \frac{N-1}{p-1} H u_\nu.
\end{equation} 
By combining \eqref{eq:00}, \eqref{eq:01}, and \eqref{eq:02} we get
\begin{equation*}
p (N-1)  u^{- \frac{p(N-1)}{N-p}} |\na u|^p  \left\lbrace \frac{- u_{\nu }}{(N-p) u} - \frac{H}{p-1} \right\rbrace \le 0,
\end{equation*}
from which, by using the fact that $u=1$ on $\Ga$, \eqref{eq:overdetermination}, \eqref{eq:capnormagradmenodernormale}, and \eqref{eq:valoreunu serrinesterno} we get
$$H_0 -H \le 0.$$
We can now conclude by using Theorem \ref{provaSBT} stated in the Introduction.
\end{proof}

\begin{rem}
{\rm
Since the solution of \eqref{Problemcapacity} in a ball is explicitly known, as a corollary of Theorem \ref{thm:Serrinesterno} we get that $u$ is spherically symmetric about the center $x_0$ of (the ball) $\Om$ and it is given by \eqref{eq:esplicitaradiale1} with $R= \frac{N-p}{(p-1)c}$.
}
\end{rem}

\begin{rem}\label{rem:Garofaoloasymptotic}
{\rm
As already mentioned before, in \cite{GS} a result slightly different from Lemma~\ref{lem:asymptoticcap} is used. In fact, in \cite[Theorem 3.1]{GS} it is proved, independently from the work \cite{KV}, that if $P$ takes its supremum at infinity, then the asymptotic expansion \eqref{eq:cap-asymptotic u} holds true and hence
$\lim_{|x| \to \infty} P(x)$ exists and it is given by \eqref{eq:prova2}; clearly, that result would be sufficient to complete the proof of Theorem \ref{thm:Serrinesterno} without invoking Lemma \ref{lem:asymptoticcap}.
}
\end{rem}

\begin{rem}
{\rm
In \cite{GS}, instead of the classical isoperimetric inequality, the authors use the nonlinear version of the isoperimetric inequality for the capacity mentioned in the Introduction stating that, for any bounded open set $A$, if $B_\rho$ denotes a ball such that $|B_\rho|= \frac{ \om_N }{N} \rho^N=|A|$,  it holds that
\begin{equation}\label{eq:Poincare}
\pCap_p (A) \ge \pCap_p (B_\rho) ,
\end{equation}
with equality if and only if $A$ is a ball (for a proof see, e.g., \cite{Ge}).

Since the $p$-capacity of a ball is known explicitly (see, e.g., \cite[Equation (4.7)]{GS}):
\begin{equation*}
\pCap_p(B_\rho)= \om_N \left( \frac{N-p}{p-1} \right)^{p-1} \rho^{N-p} ,
\end{equation*}
if we take $B_\rho$ such that $|B_\rho|= |\Om|$, for $\Om$ \eqref{eq:Poincare} becomes
\begin{equation}\label{eq:isop cap}
\pCap_p (\Om) \ge \om_N \left( \frac{N-p}{p-1} \right)^{p-1} \left( \frac{N |\Om|}{\om_N} \right)^{\frac{N-p}{N}}.
\end{equation}

Now we notice that, since \eqref{eq:cap in terms isop} holds, \eqref{eq:isop cap} is equivalent to the classical isoperimetric inequality \eqref{eq:ISOPerimetrica}.
In fact, if we put \eqref{eq:cap in terms isop} in \eqref{eq:isop cap}, it is easy to check that \eqref{eq:isop cap} becomes exactly \eqref{eq:ISOPerimetrica}.
}
\end{rem}




\section{The case \texorpdfstring{$1<p=N$}{1<p=N}: proof of Theorem \ref{thm:conforme}}\label{sec:casoconforme}

In the present section, we consider $u$ solution to the exterior problem \eqref{ConformalProblemcapacity}, \eqref{eq:overdetermination}, and $1<N$. 
In order to give the proof of Theorem \ref{thm:conforme}, we collect all the necessary ingredients in the following remark.

%
%
%

\begin{rem}
{\rm
(i)(On the regularity). We notice that the regularity results invoked in Remark \ref{rem:regularity} hold when $p=N$, too.

(ii)(Asymptotic expansion). By \eqref{eq:confasymppre}, the result of Kichenassamy and V\'eron (\cite[Remark 1.5]{KV}) applies also in this case and hence we have that \eqref{eq:passKV} holds with 
$$
\mu(x)= - \frac{ \ln{| x |} }{\om_N^{\frac{1}{N-1}}}.
$$
It is easy to check that also Identity \eqref{eq:capintegrperasymptotic} still holds (with $p$ replaced by $N$); by computing the limit in the right-hand side, and by using \eqref{eq:overdetermination} and the fact that $|\na u| = - u_\nu$ in the left-hand side, \eqref{eq:capintegrperasymptotic} leads to 
$$
c^{N-1} | \Ga|= \ga^{N-1},
$$
from which we deduce the following asymptotic expansion for $u$ at infinity:
\begin{equation}\label{eq:conformasymptoticexp}
u(x)= - c \left( \frac{|\Ga|}{\om_N} \right)^{\frac{1}{N-1}} \, \ln{|x|} + O(1).
\end{equation}
}
\end{rem}

We are ready now to prove Theorem \ref{thm:conforme}.

\begin{proof}[Proof of Theorem \ref{thm:conforme}]
We just find the analogous of the Rellich-Pohozaev-type identity \eqref{eq:provaconforme} when $p=N$; since $u$ is $N$-harmonic, now the vector field
$$X= N <x, \na u> | \na u|^{N-2} \na u - | \na u|^N x$$
is divergence-free and thus integration by parts leads to
\begin{equation}\label{eq:conformeintegrbyparts}
\int_{\Ga} <X, \nu> \, dS_x= \lim_{R \to \infty} \int_{\pa B_R} <X, \nu_{B_R}> \, dS_x.
\end{equation}

For the left-hand side in \eqref{eq:conformeintegrbyparts}, by using that $\na u = u_\nu \, \nu$ on $\Ga$ and \eqref{eq:overdetermination} we immediately find that
$$
\int_{\Ga} <X, \nu> \, dS_x= (N-1) c^N N|\Om|.
$$
For the right-hand side in \eqref{eq:conformeintegrbyparts}, by the asymptotic expansion \eqref{eq:conformasymptoticexp} we easily compute that
$$
\lim_{R \to \infty} \int_{\pa B_R} <X, \nu_{B_R}> \, dS_x = (N-1) c^N \frac{|\Ga|^{\frac{N}{N-1}}}{\om_N^{\frac{1}{N-1}}}.
$$
Thus, \eqref{eq:conformeintegrbyparts} becomes
$$
N |\Om| = \frac{|\Ga|^\frac{N}{N-1}}{\om_N^{\frac{1}{N-1}}},
$$
that is the equality case of the classical isoperimetric inequality. Hence $\Om$ must be a ball.
\end{proof}

\begin{rem}
{\rm
As a corollary of Theorem \ref{thm:conforme} we get that $u$ is spherically symmetric about the center $x_0$ of (the ball) $\Om$ and it is given by
$$
u(x)= - c \left( \frac{|\Ga|}{\om_N} \right)^{\frac{1}{N-1}} \, \ln{|x - x_0|} ,
$$
up to an additive constant.
}
\end{rem}

\section{The interior problem: proof of Theorem \ref{thm:Serrininterno}}\label{sec:Interior}

In the present section $u$ is a weak solution to \eqref{Pdelta}, \eqref{eq:interno-overdetermination}, and $1<p<N$.
By a weak solution of \eqref{Pdelta}, \eqref{eq:interno-overdetermination} we mean a function $u \in W^{1,p}_{loc} \left( \Om \setminus \left\lbrace 0 \right\rbrace \right)$ such that
$$
\int_{\Om} |\na u|^{p-2} \, \na u \cdot \na \phi \, dx=0
$$
for every $\phi \in C_0^\infty \left( \Om \setminus \left\lbrace 0 \right\rbrace \right)$ and satisfying the boundary condition in \eqref{Pdelta} and the one in \eqref{eq:interno-overdetermination} in the weak sense explained after Theorem \ref{thm:Serrinesterno}.

In order to give our proof of Theorem \ref{thm:Serrininterno}, we collect all the necessary ingredients in the following remark.

\begin{rem}
{\rm
(i)(On the regularity). The regularity results presented for the exterior problem in
Remark \ref{rem:regularity} hold in the same way also for the interior problem, so that, reasoning as explained there, we can affirm that $u$ can be extended to a $C^2$-function in a neighborhood of $\Ga$, $\Ga$ is of class $C^2$, and $u \in C^{1, \al}_{loc} (\ol{\Om})$.

(ii)(Explicit value of $K$ in \eqref{Pdelta}). It is easy to show that the normalization constant $K$ that appears in \eqref{Pdelta} must take the value
$$
K= |\Ga |,
$$
to be compatible with the overdetermined condition \eqref{eq:interno-overdetermination} (see for example \cite{EP}).

(iii)(Asymptotic expansion). As a direct application of \cite[Theorem 1.1]{KV}, we deduce that the asymptotic behaviour of the solution $u$ of \eqref{Pdelta} near the origin is given by
\begin{equation}
\label{asymptoticu}
u(x) = \frac{p-1}{N-p} \left( \frac{|\Ga|}{ \om_N} \right)^{\frac{1}{p-1}} \, |x|^{- \frac{N-p}{p-1}} + o(|x|^{- \frac{N-p}{p-1}}).
\end{equation}
We mention that this expansion has been used also in \cite{EP}.

}
\end{rem}

\vspace{1pt}
\noindent
\textbf{The P-function.}
We consider again the P-function defined in \eqref{def:P}; by virtue of the weak $p$-harmonicity of $u$ in $\Om \setminus \lbrace 0 \rbrace$, \cite[Theorem 2.2]{GS} ensures that the strong maximum principle holds for $P$ in $\Om \setminus \lbrace 0 \rbrace$.


\vspace{1pt}


We are ready now to prove Theorem \ref{thm:Serrininterno}; as announced in the Introduction, the proof uses arguments similar to those used in the proof of Theorem \ref{thm:Serrinesterno}.


\begin{proof}[Proof of Theorem \ref{thm:Serrininterno}]
As already noticed in \cite{EP}, concerning the existence of a weak solution to the overdetermined problem \eqref{Pdelta}, \eqref{eq:interno-overdetermination}, the actual value of the function $u$ on $\Ga$ is irrelevant, since any function differing from $u$ by a constant is $p$-harmonic whenever $u$ is.
Thus, let us now fix the constant $c$ that appears in \eqref{Pdelta} as
\begin{equation}\label{def:cSBT}
c= \frac{p-1}{N-p} \left( \frac{N |\Om|}{|\Ga|} \right);
\end{equation}
with this choice and by recalling \eqref{eq:interno-overdetermination} we have that
$$
P_{| \Ga} = \left( \frac{p-1}{N-p} \right)^{- \frac{p(N-1)}{N-p}} \left( \frac{N| \Om|}{| \Ga|} \right)^{- \frac{p(N-1)}{N-p}}.
$$
Moreover, by using \eqref{asymptoticu} we find also that
$$
\lim_{|x| \rightarrow 0} P(x)= \left( \frac{p-1}{N-p} \right)^{- \frac{p(N-1)}{N-p}}\left( \frac{|\Ga|}{\om_N} \right)^{- \frac{p}{N-p}}.
$$
By the isoperimetric inequality \eqref{eq:ISOPerimetrica} it is easy to check that
$$
\lim_{|x| \rightarrow 0} P(x) \le P_{| \Ga} ,
$$
and hence, by the maximum principle proved in \cite[Theorem 2.2]{GS}, we realize that $P$ attains its maximum on $\Ga$. We thus have that
\begin{equation}\label{eq:MPtoP}
0\le P_{\nu}.
\end{equation}
By a direct computation, with exactly the same manipulations used for the exterior problem in the proof of Theorem \ref{thm:Serrinesterno}, we find that
\begin{equation}\label{eq:PnuSBT}
P_{\nu} = p (N-1)  u^{- \frac{p(N-1)}{N-p}} |\na u|^p  \left\lbrace \frac{- u_{\nu}}{(N-p) u} - \frac{H}{p-1} \right\rbrace.
\end{equation}

By coupling \eqref{eq:MPtoP} with \eqref{eq:PnuSBT}, we can deduce that
\begin{equation*}
H \le \frac{p-1}{N-p} \left(\frac{- u_\nu}{u}\right),
\end{equation*}
that by using \eqref{def:cSBT}, \eqref{Pdelta}, \eqref{eq:interno-overdetermination}, and the fact that $|\na u|= - u_\nu$ on $\Ga$, leads to
\begin{equation}\label{eq:penultima SBT}
H \le H_0,
\end{equation}
where $H_0$ is the constant defined in \eqref{def:H0}.

Since $\Om$ is star-shaped with respect to a point $z \in \Om$ (possibly distinct from $0$), we have that $<(x - z) ,\nu >$ is non-negative on $\Ga$.
Thus, multiplying \eqref{eq:penultima SBT} by $<(x - z) ,\nu>$,
%
%
and integrating over $\Ga$, we get
\begin{equation*}
\int_\Ga H <(x - z ) ,\nu> \, dS_x \le |\Ga|.
\end{equation*}
By recalling the regularity of $\Ga$ and {\it Minkowski's identity}
\begin{equation*}
\int_\Ga H <(x - z ) ,\nu> \, dS_x = |\Ga| ,
\end{equation*}
we deduce -- as already noticed in \cite[Proof of Theorem 1.1]{GS} -- that the equality sign must hold in \eqref{eq:penultima SBT}, that is
$$
H \equiv H_0.
$$
Thus, the conclusion follows by the classical Alexandrov's Soap Bubble Theorem (\cite{Al}) or, if we want, again by Theorem \ref{provaSBT}.
\end{proof}

\begin{rem}
{\rm
As a corollary of Theorem \ref{thm:Serrininterno} we get that $u$ is spherically symmetric about the center $0$ of (the ball) $\Om$ and it is given by
$$
u(x) = \frac{p-1}{N-p} \left( \frac{|\Ga|}{ \om_N} \right)^{\frac{1}{p-1}} \, |x|^{- \frac{N-p}{p-1}} ,
$$
up to an additive constant.
}
\end{rem}

\section*{Acknowledgements}
The author wishes to thank Rolando Magnanini for his constructive criticism and for pointing out the paper \cite{AR}.
The author also wishes to thank Nicola Garofalo for his kind interest in this work, and Chiara Bianchini, Andrea Colesanti, and Paolo Salani for bringing to his attention the references \cite{EP}, \cite{CoS}, and \cite{CC}.

The paper was partially supported by Gruppo Nazionale Analisi Matematica Probabilit\`a e Applicazioni (GNAMPA) of the Istituto Nazionale di Alta Matematica (INdAM).

\end{document}